\documentclass[a4paper,12pt]{article}
\usepackage{comment,amsmath,amssymb,amsthm,pstricks,changepage,pifont,graphicx}
\usepackage[english]{babel}

\newtheorem{theorem}{Theorem}

\newtheorem{lemma}[theorem]{Lemma}

\newcommand{\conv}{\text{conv}}

\begin{document}

\title{A minimal set of generators for the canonical ideal of a non-degenerate curve}
\author{Wouter Castryck\footnote{Supported by F.W.O.-Vlaanderen.} ~and Filip Cools}
\date{}

\maketitle

\begin{abstract}
  \noindent We give an explicit way of writing down a minimal set of generators 
  for the canonical ideal of a non-degenerate curve, or of a more general smooth projective curve in
  a toric surface, in terms of its defining Laurent polynomial.\\
  %Our generators can be seen as a two-dimensional analogue of (and in the trigonal case specialize to) 
  %Reid's rolling factors. 
  %The efficiency of the method is illustrated by an implementation in Magma.\\
  %This can be used to prove 
  %that the Newton polygon of $f$ can be recovered (up to unimodular equivalence) from
  %the abstract geometry of $C_f$.\\
  

\noindent Accompanying Magma file: \texttt{canonical.m}
\end{abstract}

\section{Introduction}

Let $k$ be an algebraically closed field and consider the affine torus
$\mathbb{T}^2 = \left( k \setminus \{0\} \right)^2$. Let $\Delta \subset \mathbb{R}^2$
be a two-dimensional lattice polygon and define $N = \sharp( \Delta \cap \mathbb{Z}^2)$. 
In this article we are concerned
with algebraic curves $U_f \subset \mathbb{T}^2$ that are
cut out by a sufficiently generic Laurent polynomial 
\[ f = \sum_{(i,j) \in \Delta \cap \mathbb{Z}^2} c_{i,j} x^iy^j \, \in k[x^{\pm 1}, y^{\pm 1}]. \] 
Here `sufficiently generic' means
that $f$ is contained in a certain Zariski dense subset of the corresponding $N$-dimensional
coefficient space. More precisely, to each $(i,j) \in \Delta \cap \mathbb{Z}^2$ 
we associate a formal variable $X_{i,j}$, and we let
\[ \mathbb{P}^{N-1} = \text{Proj}\, k[X_{i,j}]_{(i,j) \in \Delta \cap \mathbb{Z}^2}. \]
We have a natural embedding
\[ \varphi_\Delta : \mathbb{T}^2 \hookrightarrow \mathbb{P}^{N-1} : (x,y) \mapsto \left( x^iy^j \right)_{(i,j) \in \Delta \cap \mathbb{Z}^2}, \]
the Zariski closure of the image of which is a toric surface 
that we denote by $\text{Tor}(\Delta)$. 
Note that $\varphi_\Delta(U_f)$ is contained in the hyperplane section
\[ H : \sum_{(i,j) \in \Delta \cap \mathbb{Z}^2} c_{i,j} X_{i,j} = 0\]
of $\text{Tor}(\Delta) \subset \mathbb{P}^{N-1}$.
Then by `sufficiently generic' we mean that the Zariski closure $C_f$ of 
$\varphi_\Delta(U_f)$ is a smooth projective curve that equals this hyperplane section. Bertini's theorem
implies that this is indeed a Zariski dense condition. Alternatively and more explicitly, 
for $C_f$ to arise as a smooth hyperplane section of $\text{Tor}(\Delta)$, it suffices that $f$ is \emph{non-degenerate} with respect 
to $\Delta$, in the sense that
for each face $\tau \subset \Delta$ (vertex, edge, or $\Delta$ itself) 
the system
\[ f_\tau = \frac{\partial f_\tau}{\partial x} = \frac{\partial f_\tau}{\partial y} = 0 \]
has no solutions in $\mathbb{T}^2$. Here $f_\tau$ is obtained from $f$ by restricting to those terms that are supported on $\tau$.
Non-degeneracy is known to be generically satisfied; see \cite[Prop.\,1]{CDV}.\\

\noindent \emph{Remark.} Every (nef and big) smooth projective curve $C$ on a toric surface $X$ 
arises as such a toric hyperplane section. Indeed, let $D_C$ be a torus-invariant divisor on $X$ that is linearly equivalent to $C$, and
let $\Delta$ be the two-dimensional lattice polygon associated to $D_C$ (here we use that $C$ is nef and big).
Then the $\mathbb{T}^2$-part of $C$ is cut out by a Laurent polynomial $f \in k[x^{\pm 1}, y^{\pm 1}]$
that is supported on $\Delta$. The above construction then yields a hyperplane section $C_f$ of $\text{Tor}(\Delta)$
that is isomorphic to $C$.\\

\noindent We refer to \cite[\S3-4]{CaCo1} and the references therein for more background, both on curves in toric surfaces
and on non-degenerate Laurent polynomials. Various of these references assume the base field $k$ to be of characteristic $0$,
but we emphasize that the material presented below is valid in any characteristic.\\

\noindent The main result of this paper is an explicit recipe for writing down a minimal set of generators
for the canonical ideal of curves of the form $C_f$, where $f \in k[x^{\pm 1}, y^{\pm 1}]$ 
satisfies the above generic condition (e.g.\ non-degeneracy) with respect to a given
two-dimensional lattice polygon $\Delta$.\\
%The input to our problem is the pair $(f, \Delta)$.\\
%Alternatively, one just inputs $f$ and lets $\Delta = \Delta(f)$ be its Newton polygon.\\

\noindent A quick implementation of the resulting algorithm already heavily outperforms 
Magma's built-in function for computing canonical ideals \cite{magma}.
The latter
relies on general lattice basis reduction algorithms that were developed by Hess \cite{Hess}.
Our code can be
found in the file \texttt{canonical.m} that accompanies the article.
It allows one to compute the canonical ideal of a non-degenerate curve of genus $g \approx 100$
in a matter of minutes, whereas everything beyond $g = 20$ looks hopeless using the Magma intrinsic,
both in terms of time and memory.
Of course, this comes at the cost of working in less generality, but note that the condition of non-degeneracy
is generically satisfied (for a fixed instance of $\Delta$), and easy to verify. 
It therefore seems useful to begin the computation of the canonical ideal
with a test for whether the input polynomial is non-degenerate or not, and if yes, to proceed
with the method presented here.\\ 

\noindent Our starting point is a theorem by Khovanskii~\cite{Khovanskii}, stating that
there exists a canonical divisor $K_\Delta$ on $C_f$ such that a basis
for $H^0(C_f, K_\Delta)$ is given by 
\[ \left\{ x^i y^j \right\}_{(i,j) \in \Delta^{(1)} \cap \mathbb{Z}^2}, \]
where $\Delta^{(1)}$ denotes the convex hull of the interior lattice points of $\Delta$. Here $x,y$
are viewed as functions on $C_f$ through $\varphi_\Delta$. See \cite[Prop.\,10.5.8]{coxlittleschenck} for a modern proof. 
Two notable corollaries are:
\begin{itemize}
  \item The genus of $C_f$ equals $g = \sharp(\Delta^{(1)} \cap \mathbb{Z}^2)$.
  \item If $g \geq 2$ then the linear system $|K_\Delta|$ maps $U_f$ inside the image of $\varphi_{\Delta^{(1)}}$. In particular:
  \begin{itemize}
    \item if $\Delta^{(1)}$ is one-dimensional then the canonical image of $C_f$ is a rational normal curve of
    degree $g-1$, hence $C_f$ is hyperelliptic;
    \item if $\Delta^{(1)}$ is two-dimensional, then $C_f$ is non-hyperelliptic and the canonical image of $C_f$ is contained in the toric
    surface $\text{Tor}(\Delta^{(1)}) \subset \mathbb{P}^{g-1}$.
  \end{itemize}
\end{itemize}
\noindent See \cite[\S4]{CaCo1} and its references for more details.\\

\noindent In what follows we assume that $C_f$ is non-hyperelliptic or, equivalently, 
that $\Delta^{(1)}$ is two-dimensional. Then the generators for the canonical ideal of $C_f$
are gathered in two steps.
\begin{itemize}
\item In Section~\ref{section_toricideal}, which can be
seen as an addendum to previous work by Koelman \cite{Koelman1,Koelman2},
we will describe a method for finding 
a minimal set of generators for the ideal of $\text{Tor}(\Delta^{(1)})$. We also provide
explicit formulas for the number of generators in each degree.
Because of the independent interest, we will do this
for toric surfaces $\text{Tor}(\Gamma)$ where $\Gamma$ is an \emph{arbitrary}
two-dimensional lattice polygon (not necessarily of the form $\Delta^{(1)}$).

\item Then in Section~\ref{section_ideal}, we will explicitly describe which generators have to be added
in order to obtain a minimal set of generators for the canonical ideal of $C_f$. These can be
seen as analogues of Reid's rolling factors~\cite{Stevens}, where the `rolling' now happens in two directions, rather than one.
\end{itemize}

\noindent \emph{Notation and terminology.} We use a special notation for two recurring polygons
\[ \Sigma = \conv \{ (0,0), (1,0), (0,1) \}, \qquad \qquad \Upsilon = \conv \{ (-1,-1), (1,0), (0,1) \}, \] 
and write $\cong$ to indicate unimodular equivalence. For instance, 
$\Delta \cong \Sigma$ if and only if $\Delta$ is a unimodular simplex.
We recall that the convex hull of the interior lattice points of
a two-dimensional lattice polygon $\Delta$ is denoted by $\Delta^{(1)}$. If
the latter is again two-dimensional, we abbreviate $\Delta^{(1)(1)}$ by $\Delta^{(2)}$.
We use $\Delta^\circ$ to denote the topological interior of $\Delta$, and write $\partial \Delta$ for
its boundary.
A two-dimensional lattice polygon $\Delta$ is said to be \emph{hyperelliptic} if $\Delta^{(1)}$ 
is one-dimensional. If $X$ is
a projectively embedded variety over $k$, we write $\mathcal{I}(X)$ for its defining ideal. For each non-negative integer
$d$ we use $\mathcal{I}_d(X)$ to denote the $k$-vector space of homogeneous degree $d$ polynomials
that are contained in $\mathcal{I}(X)$.

\section{The ideal of a toric surface} \label{section_toricideal}

\noindent Let $\Gamma \subset \mathbb{R}^2$ be a two-dimensional 
lattice polygon and let $N = \# (\Gamma \cap \mathbb{Z}^2)$. Define
$\text{Tor}(\Gamma)$ as the Zariski closure inside $\mathbb{P}^{N-1}$ of the image of $\varphi_\Gamma$.
A result due to Koelman \cite{Koelman1,Koelman2} states that the ideal $\mathcal{I}(\text{Tor}(\Gamma))$ is generated by all binomials
\[ \prod_{\ell=1}^n X_{i_\ell,j_\ell} - \prod_{\ell=1}^n X_{i'_\ell,j'_\ell} \qquad \text{for which} \qquad \sum_{\ell=1}^n (i_\ell,j_\ell) = \sum_{\ell=1}^n (i'_\ell,j'_\ell) \]
where $n \in \{2,3\}$. Moreover one can restrict to $n=2$ if and only if $\sharp (\partial \Gamma \cap \mathbb{Z}^2) \geq 4$ or
$\Gamma$ is a unimodular simplex. This result was generalized to property $N_p$ for arbitrary $p$ by Hering and Schenck; see \cite[Thm.\,4.20]{Hering}.\\

\noindent The current section can be seen as an addendum to Koelman's work: we give explicit formulas for the \emph{number} of quadrics
and cubics in a minimal set of homogeneous generators for $\mathcal{I}(\text{Tor}(\Gamma) )$.

\begin{lemma} \label{toricdegreed}
 For all integers $d \geq 0$ one has:
 $$\dim \mathcal{I}_d \left(\emph{Tor}(\Gamma) \right)={ \sharp (\Gamma \cap \mathbb{Z}^2) +d - 1 \choose d}-\sharp (d\Gamma \cap \mathbb{Z}^2).$$
\end{lemma}

\begin{proof}
The $k$-vector space morphism
$$ \chi_d : \mathcal{I}_d(\mathbb{P}^{N-1}) \rightarrow k[x^{\pm 1}, y^{\pm 1}] : X_{i_1,j_1}\cdots X_{i_d,j_d} \mapsto x^{i_1+\cdots+i_d}y^{j_1+\cdots +j_d} $$
has kernel $\mathcal{I}_d (\text{Tor}(\Gamma) )$ and surjects onto 
$\langle x^iy^j \rangle_{(i,j) \in d\Gamma \cap \mathbb{Z}^2}$ (here we use that 
two-dimensional lattice polygons are always normal \cite[Prop.\,1.2.2-4]{bruns}, i.e.\ every lattice point
in $d\Gamma$ is the sum of $d$ lattice points in $\Gamma$).
\end{proof}

\noindent The main result of this section is:

\begin{theorem}
 A minimal set of generators for $\mathcal{I}(\emph{Tor}(\Gamma) )$ consists of
 \[ { \sharp (\Gamma \cap \mathbb{Z}^2) + 1 \choose 2}-\sharp (2\Gamma \cap \mathbb{Z}^2) \text{ quadrics and $c_\Gamma$ cubics,}\]
 where
 \[ c_\Gamma = \left\{ \begin{array}{ll} 0 & \text{if $\sharp (\partial \Gamma \cap \mathbb{Z}^2) \geq 4$ or $\Gamma \cong \Sigma$,} \\ 
                                         1 & \text{if $\sharp (\partial \Gamma \cap \mathbb{Z}^2) = 3$, $\Gamma \not \cong \Sigma$, and $\Gamma$ is non-hyperelliptic,} \\
                                         \sharp (\Gamma \cap \mathbb{Z}^2) - 3 & \text{if $\sharp (\partial \Gamma \cap \mathbb{Z}^2) = 3$, $\Gamma \not \cong \Sigma$, and $\Gamma$ is hyperelliptic.} \\ \end{array} \right. \]

 \end{theorem}
\begin{proof}
 The formula for the number of quadrics follows from Lemma~\ref{toricdegreed} along with the fact
 that $\text{Tor}(\Gamma)$ is not contained in any hyperplane of $\mathbb{P}^{N-1}$. By Koelman's result,
 it remains to prove the formula for the number of cubics $c_\Gamma$ when $\sharp (\partial \Gamma \cap \mathbb{Z}^2) = 3$
 and $\Gamma \not \cong \Sigma$. We
 moreover know that $c_\Gamma \geq 1$ in these cases. Also recall that 
 $\mathcal{I}(\text{Tor}(\Gamma))$ is generated by binomials.\\
 
 \noindent First assume that $\Gamma$ is non-hyperelliptic and $\Gamma \not \cong \Upsilon$. Along with
 $\sharp (\partial \Gamma \cap \mathbb{Z}^2) = 3$ and $\Gamma \not \cong \Sigma$ this implies
 that $\Gamma^{(1)}$ is two-dimensional; see e.g.\ Koelman's classification
 \cite[Ch.\,4]{Koelman}, although this could also serve as an easy exercise.
 Let $\{v_1,v_2,v_3\}$ be the three vertices of $\Gamma$
 and consider
 \[ \Gamma' = \conv \left( (\Delta \setminus \{ v_1 \} ) \cap \mathbb{Z}^2 \right).\]
 Then $\Gamma' \supset \Gamma^{(1)}$ is again a two-dimensional lattice polygon. 
 We claim that there are at least $4$ lattice points on its boundary. 
 Indeed, if there would only be $3$ such lattice points, then $\Gamma'$ would be a triangle whose vertices are $\{ v, v_2, v_3\}$, where
 $v$ is contained in the interior of $\Gamma$,
 and the triangles $v_1$-$v$-$v_2$ and $v_1$-$v$-$v_3$ are unimodular simplices (i.e.\ they do 
 not contain any lattice points besides the vertices).
\begin{center}
    \includegraphics[scale=0.9]{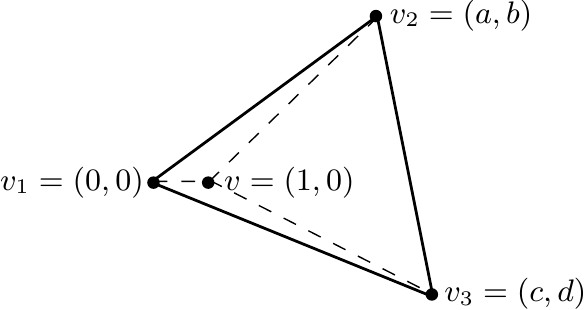}
\end{center}
 We may assume that $v_1 = (0,0)$, $v = (1,0)$, $v_2 = (a,b)$ and $v_3 = (c,d)$, where
 $b > 0 > d$. 
By Pick's theorem the unimodularity of $v_1$-$v$-$v_2$ and $v_1$-$v$-$v_3$ 
 implies that $b = 1$ and $d = -1$, and hence that $\Gamma$ 
 is contained in a horizontal strip of width $2$: a contradiction with
 the fact that $\Gamma^{(1)}$ is two-dimensional. So the claim follows.
 Now consider a binomial
 \begin{equation} \label{cubicbin}
  C = X_{i_1,j_1} X_{i_2,j_2} X_{i_3,j_3} - X_{i'_1,j'_1} X_{i'_2,j'_2} X_{i'_3,j'_3} \in \mathcal{I}_3(\text{Tor}(\Gamma))
 \end{equation}
 and define $\Gamma_C = \conv \{(i_1,j_1), (i_2,j_2), (i_3,j_3), (i_1',j_1'), (i_2',j_2'), (i_3',j_3') \}$.
 \begin{itemize}
  \item If $\Gamma_C \subsetneq \Gamma$, then by the above $\Gamma_C \subset \Gamma'$ for a subpolygon
  $\Gamma'$ that contains at least $4$ lattice points on the boundary. So by Koelman's result applied
  to $\Gamma'$ our cubic $C$ can be written
  as a linear combination of a number of elements of $\mathcal{I}_2(\text{Tor}(\Gamma))$.
  \item If $\Gamma_C = \Gamma$ then it is not hard to see that either $(i_1,j_1), (i_2,j_2), (i_3,j_3)$ 
  or $(i'_1,j'_1), (i'_2,j'_2), (i'_3,j'_3)$ are the
  three vertices of $\Gamma$; see \cite[Lem.\,2.6]{Koelman2}. 
 \end{itemize}
 It follows that the sum or difference of two binomials $C_1,C_2
 \in \mathcal{I}_3(\text{Tor}(\Gamma))$ that are independent of $\mathcal{I}_2(\text{Tor}(\Gamma))$ 
 is again a cubic binomial $C$. But the latter satisfies $\Gamma_{C} \subsetneq \Gamma$,
 so by the first observation $C$ is expressible as a linear combination
 of elements of $\mathcal{I}_2(\text{Tor}(\Gamma))$. This proves that one cubic is sufficient, i.e.\ $c_\Gamma = 1$.\\
 
 \noindent Next assume that $\Gamma$ is hyperelliptic or $\Gamma \cong \Upsilon$. Using that
 $\sharp (\partial \Gamma \cap \mathbb{Z}^2) = 3$ we find that it is unimodularly equivalent to
\begin{center}
    \includegraphics[scale=0.9]{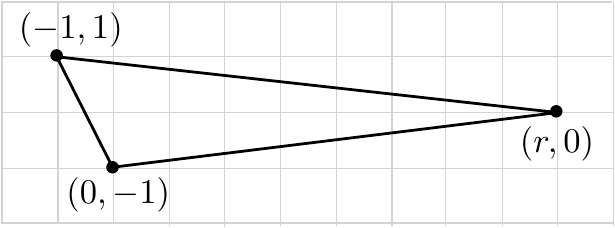}
\end{center} 

 \noindent where $r = \# (\Gamma \cap \mathbb{Z}^2) - 3$. 
 One verifies that the irreducible binomials in 
 $\mathcal{I}_3(\text{Tor}(\Gamma))$ involving $X_{-1,1}$ or $X_{0,-1}$ must involve
 both variables in the same monomial. This monomial is necessarily among
 \[ X_{-1,1}X_{0,-1}X_{i,0} \qquad i = 1, \dots, r \] 
 and conversely, for each of these monomials there is a corresponding binomial $C_i \in
 \mathcal{I}_3(\text{Tor}(\Gamma))$. As before we find that
 the difference or sum of
 two cubic binomials involving the same monomial $X_{-1,1}X_{0,-1}X_{i,0}$ is a linear combination of
 elements of $\mathcal{I}_2(\text{Tor}(\Gamma))$. So we conclude that
 $\mathcal{I}(\text{Tor}(\Gamma))$ is generated by
 $\mathcal{I}_2(\text{Tor}(\Gamma)) \cup \{ C_1, \dots, C_r \}$.
 Because the quadratic binomials in $\mathcal{I}(\text{Tor}(\Gamma))$
 do neither involve $X_{-1,1}$ nor $X_{0,-1}$, the latter $r$ cubics are independent of $\mathcal{I}_2(\text{Tor}(\Gamma))$.
\end{proof}

%\begin{corollary} \label{toricbettidiagram}
% The Betti diagram of $\emph{Tor}(\Gamma)$ is of the form
%\end{corollary}

\vspace{0.2cm}
\noindent We have included Magma code for computing such a minimal set of (binomial) generators; see
our accompanying file \verb"canonical.m". As for
the quadratic generators, this is done by naively gathering all relations of the form
\[ (i_1,j_1) + (i_2,j_2) = (i'_1,j'_1) + (i'_2,j'_2) \]
for $(i_1,j_1),(i_2,j_2),(i'_1,j'_1),(i'_2,j'_2) \in \Gamma \cap \mathbb{Z}^2$,
and then finding a $k$-linearly independent subset of the set of corresponding
binomials
\[ X_{i_1,j_1}X_{i_2,j_2} - X_{i'_1,j'_1}X_{i'_2,j'_2}. \]
In the case where $\sharp \left( \partial \Gamma \cap \mathbb{Z}^2 \right) = 3$, $\Gamma \not \cong \Sigma$ and
$\Gamma$ is non-hyperelliptic, a single binomial of the form (\ref{cubicbin}) with 
$(i_1,j_1),(i_2,j_2), (i_3,j_3)$ the vertices of $\Gamma$ is added
by exhaustive search. In the hyperelliptic case the explicit construction from
the above proof is followed.\\

\noindent \emph{Example.} The code below
carries this out for the following lattice polygon (over $k = \overline{\mathbb{Q}}$):
\begin{center}
    \includegraphics[scale=0.9]{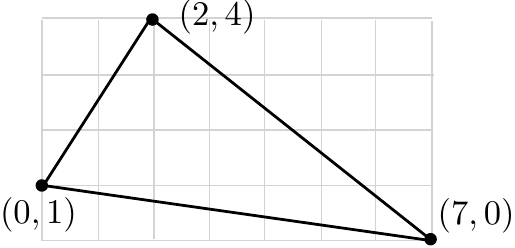}
\end{center}
\noindent \texttt{> load "canonical.m"}\\
\noindent \texttt{Loading "canonical.m"}\\
\noindent \texttt{Loading "basic\char`_commands.m"}\\
\noindent \verb"> P := LatticePolytope([<0,1>,<7,0>,<2,4>]);"\\
\noindent \verb"> time I := TorIdeal(P, Rationals());"\\
\noindent \verb"Time: 0.110"\\
\noindent This can be used as input to more advanced functions, such as the Magma intrinsic for
computing the Betti diagram:\\
\noindent \verb"> BettiTable(GradedModule(Ideal(I)));"\\
\noindent \verb"["\\
\noindent \verb"    [ 1, 0, 0, 0, 0, 0, 0, 0, 0, 0, 0, 0 ],"\\
\noindent \verb"    [ 0, 55, 320, 891, 1424, 1470, 972, 315, 16, 0, 0, 0 ],"\\
\noindent \verb"    [ 0, 1, 11, 71, 480, 1302, 1932, 1886, 1221, 485, 110, 11 ]"\\
\noindent \verb"]"\\

\noindent \emph{Remark.} From the point of view of efficiency the above method leaves room for improvement. 
Especially the gathering of the quadratic generators can be 
done more systematically, for instance using Gr\"obner bases computations. These 
are implicitly invoked by the code below (a continuation of
the above example):\\
\noindent \verb"> AA<x,y> := AffinePlane(Rationals());"\\
\noindent \verb"> latticepoints := ConvexHull(P); N := #latticepoints;"\\
\noindent \verb"> PP := ProjectiveSpace(Rationals(), N-1);"\\
\noindent \verb"> phi_P := map< AA->PP | [x^p[1]*y^p[2] : p in latticepoints] >;"\\
\noindent \verb"> time I := Ideal(Image(phi_P));"\\
\noindent \verb"Time: 0.080"\\
This produces a reduced Gr\"obner basis for $\text{Tor}(\Gamma)$. In general this is \emph{not} a minimal set 
of generators, but its quadratic elements do form a basis of $\mathcal{I}_2(\text{Tor}(\Gamma))$, so that one
can obtain a minimal set of generators by proceeding as above.\\

\noindent \emph{Remark.} Up to unimodular equivalence, the only two-dimensional instances
of $\Delta^{(1)}$ for which $\sharp(\partial \Delta^{(1)} \cap \mathbb{Z}^2) = 3$ are $\Sigma$ and $\Upsilon$.
This can be shown using \cite[Lem.\,9-11]{HaaseSchicho}. Therefore, for the purposes of describing the canonical
ideal of curves in toric surfaces, the above general treatment is more elaborate than needed. We have included
it because we believe it to be of independent interest.

\section{An explicit description of the canonical ideal} \label{section_ideal}

Let $\Delta$ be a two-dimensional lattice polygon and 
let $f \in k[x^{\pm 1}, y^{\pm 1}]$ be a Laurent polynomial
satisfying the sufficiently generic condition from the introduction (e.g.\ non-degeneracy).
Assume that the corresponding curve $C_f$ is non-hyperelliptic of genus $g \geq 3$, i.e.\
$\Delta^{(1)}$ is two-dimensional and $\sharp (\partial \Delta^{(1)} \cap \mathbb{Z}^2) \geq 3$.
Let $C_f^\text{can}$ be the canonical model of $C_f$ obtained using $|K_\Delta|$.\\

\noindent We already know that $\mathcal{I}(C_f^\text{can})$ contains 
$\mathcal{I}(\text{Tor}(\Delta^{(1)}))$, and from the previous section
we know how to find a minimal set of generators for the latter. 
In this section we describe which generators have to be added in order
to obtain a minimal set of generators for $\mathcal{I}(C_f^\text{can})$.
A priori, it is not entirely trivial that it suffices to merely \emph{add} some
generators, but note from the previous remark that $\text{Tor}(\Delta^{(1)})$ is almost always generated
by quadrics, in which case this is clear. The only exception is when $\Delta^{(1)} \cong \Upsilon$,
which corresponds to curves of genus $4$, and is therefore well-understood.\\

\noindent Our main auxiliary tool is:

\begin{theorem} \label{thm_dimIC}
The equality 
$$\dim \mathcal{I}_d(C_f^\emph{can}) -\dim \mathcal{I}_d(\emph{Tor}(\Delta^{(1)})) = \sharp \left(\left((d-1)\Delta^{(1)}\right)^{(1)}\cap \mathbb{Z}^2\right)$$
holds for all integers $d\geq 2$.
\end{theorem}

\begin{proof}
From Lemma~\ref{toricdegreed} it follows that
$$\dim \mathcal{I}_d(\text{Tor}(\Delta^{(1)}))={g+d-1\choose d}-\sharp (d\Delta^{(1)} \cap \mathbb{Z}^2).$$
On the other hand, let $H(d)$ be the Hilbert function of the homogeneous coordinate ring of $C_f^\text{can} \subset \mathbb{P}^{g-1}$. Then  
$H(d)=(2g-2)d+(1-g)=(2d-1)(g-1)$ if $d\geq 2$ (see \cite[Cor.~9.4]{eisenbud}), hence 
$$\dim \mathcal{I}_d(C)={g+d-1\choose d}-(2d-1)(g-1).$$ 
So we are left with proving that $$\sharp (d\Delta^{(1)} \cap \mathbb{Z}^2)-\sharp \left(\left((d-1)\Delta^{(1)}\right)^{(1)}\cap \mathbb{Z}^2\right)=(2d-1)(g-1).$$
For this, write $R^{(1)}=\sharp (\partial\Delta^{(1)}\cap\mathbb{Z}^2)$ and consider the Ehrhart polynomial $$\text{Ehr}_{\Delta^{(1)}}(k)=\text{Vol}(\Delta^{(1)})\cdot k^2+\frac{R^{(1)}}{2}\cdot k+1$$ of $\Delta^{(1)}$. Since $\sharp(k\Delta^{(1)}\cap \mathbb{Z}^2)=\text{Ehr}_{\Delta^{(1)}}(k)$ and $\sharp \left(\partial \left(k\Delta^{(1)}\right)\cap \mathbb{Z}^2\right)=kR^{(1)}$ for all $k\in\mathbb{Z}_{\geq 1}$, we have that 
\begin{eqnarray*}
\lefteqn{\sharp (d\Delta^{(1)} \cap \mathbb{Z}^2)-\sharp \left(\left((d-1)\Delta^{(1)}\right)^{(1)}\cap \mathbb{Z}^2\right)}\\
&=& \text{Ehr}_{\Delta^{(1)}}(d) - \text{Ehr}_{\Delta^{(1)}}(d-1) + \sharp \left(\partial \left((d-1)\Delta^{(1)}\right)\cap \mathbb{Z}^2\right) \\
&=& (2d-1)\left(\text{Vol}(\Delta^{(1)})+\frac{R^{(1)}}{2}\right) \\
&=& (2d-1)(g-1).
\end{eqnarray*}
This concludes the proof. \end{proof}

\noindent \emph{Remark.} Some readers may prefer the following cohomological
proof of Theorem~\ref{thm_dimIC} (brief). Assume for ease of exposition
that $\text{Tor}(\Delta)$ is smooth; if not the argument below has to be preceded by a toric blow-up. 
Let $D_{C_f}$ be
a torus-invariant divisor on $\text{Tor}(\Delta)$ that is linearly equivalent to $C_f$, let $K$ be a torus-invariant
canonical divisor on $\text{Tor}(\Delta)$, and define $L = D_{C_f} + K$. 
When tensoring the exact sequence
\[ 0 \rightarrow \mathcal{O}_{\text{Tor}(\Delta)}(-D_{C_f}) \rightarrow \mathcal{O}_{\text{Tor}(\Delta)} \rightarrow \mathcal{O}_{C_f} \rightarrow 0 \]
with $\mathcal{O}_{\text{Tor}(\Delta)}(dL)$, taking cohomology and using the standard toric vanishing
theorems for $H^1$ we get
\[ 0 \rightarrow H^0(\text{Tor}(\Delta), (d-1)L + K) \rightarrow H^0(\text{Tor}(\Delta), dL) \rightarrow H^0(C_f, dL|_{C_f}) \rightarrow 0. \]
The respective dimensions of these spaces are seen to be
\[ \sharp \left(\left((d-1)\Delta^{(1)}\right)^{(1)}\cap \mathbb{Z}^2\right), \ \dim \frac{\mathcal{I}_d(\mathbb{P}^{g-1})}{\mathcal{I}_d(\text{Tor}(\Delta^{(1)}))}, \text{ and } \dim \frac{\mathcal{I}_d(\mathbb{P}^{g-1})}{\mathcal{I}_d(C_f^\text{can})}, \]
(indeed, by adjunction theory $L|_{C_f}$ is a canonical divisor on $C_f$), so that the theorem follows.\\

\noindent Now write
$$ f = \sum_{(i,j) \in \Delta \cap \mathbb{Z}^2} c_{i,j} x^iy^j \in k[x^{\pm 1},y^{\pm 1}]$$
and define $\mathcal{W}_d=\left(\Delta^{(1)}\right)^\circ\cap\left(\frac{1}{d-1}\mathbb{Z}\right)^2$. Note that 
$$\sharp\, \mathcal{W}_d=\sharp \left(\left((d-1)\Delta^{(1)}\right)^{(1)}\cap \mathbb{Z}^2\right).$$ 
To every $w\in \mathcal{W}_d$ we can associate a homogeneous degree $d$ polynomial, as follows. 
For each $(i,j)\in\Delta\cap\mathbb{Z}^2$ there exist $$v_{1,(i,j)},\ldots,v_{d,(i,j)}\in\Delta^{(1)}\cap\mathbb{Z}^2$$ such that 
\begin{equation} \label{normality}
  (i,j) - w = (v_{1,(i,j)} - w) + \ldots + (v_{d,(i,j)} - w).
\end{equation}
This follows from the inclusion $\left((d-1)\Delta^{(1)}\right)^{(1)}+\Delta\subset d\Delta^{(1)}$ and the normality of the polygon $\Delta^{(1)}$.  
The $d$-form $$\mathcal{F}_{d,w}=\sum_{(i,j)\in\Delta\cap\mathbb{Z}^2}\,c_{i,j}X_{v_{1,(i,j)}}\cdots X_{v_{d,(i,j)}}$$
is well-defined modulo the ideal of $\text{Tor}(\Delta^{(1)})$. It clearly vanishes on $\varphi_{\Delta^{(1)}}(U_f)$, 
hence it is contained in the ideal of $C_f^\text{can}$.\\

\noindent The forms $\mathcal{F}_{d,w}$ with $w \in \mathcal{W}_d$ are $k$-linearly independent of 
each other and of the forms in $\mathcal{I}_d(\text{Tor}(\Delta^{(1)}))$. Indeed, this holds because
\[ \chi_d(\mathcal{F}_{d,w}) = (x,y)^{(d-1)w} \cdot f,\]
where $\chi_d$ is the vector space morphism from the proof of Lemma~\ref{toricdegreed}.
Hence any linear combination in which the $\mathcal{F}_{d,w}$'s appear non-trivially
is mapped to a non-zero multiple of $f$, and must therefore be non-zero itself.
By Theorem \ref{thm_dimIC}, we can conclude that 
a basis for $\mathcal{I}_d(C_f^\text{can})$ is obtained by adjoining
$ \left\{ \mathcal{F}_{d,w} \right\}_{w\in\mathcal{W}_d}$ to a basis
for $\mathcal{I}_d(\text{Tor}(\Delta^{(1)}))$. In other words:
\begin{equation} \label{claimId}
\mathcal{I}_d(C_f^\text{can})=\mathcal{I}_d(\text{Tor}(\Delta^{(1)})) \oplus \langle \mathcal{F}_{d,w}\rangle_{w\in\mathcal{W}_d}.\\
\end{equation}

\noindent We are now ready to prove our main theorem.

\begin{theorem} \label{thmcanonical}
Let $\Delta$ be a two-dimensional lattice polygon and 
let $f \in k[x^{\pm 1}, y^{\pm 1}]$ be a Laurent polynomial
satisfying the sufficiently generic condition from the introduction (e.g.\ non-degeneracy).
Assume that $\Delta^{(1)}$ is two-dimensional and let $g = \sharp (\Delta^{(1)} \cap \mathbb{Z}^2)$.
\begin{itemize}
  \item If $\Delta^{(2)} \neq \emptyset$ and $\Delta^{(1)}\not\cong\Upsilon$, then a minimal set of
  generators for $\mathcal{I}(C_f^\emph{can})$ is given by 
  a basis for $\mathcal{I}_2(\emph{Tor}(\Delta^{(1)}))$ and the quadrics $\{\mathcal{F}_{2,w} \}_{w \in \Delta^{(2)} \cap \mathbb{Z}^2}$. 
  \item If $\Delta^{(1)}\cong \Upsilon$ then a minimal set of generators for $\mathcal{I}(C_f^\emph{can})$ 
  is given by the cubic defining $\emph{Tor}(\Delta^{(1)}) \subset \mathbb{P}^3$ and the quadric $\mathcal{F}_{2,w}$ with $\Delta^{(2)}=\{w\}$.
  \item If $\Delta^{(1)}\cong\Sigma$ then a minimal set of generators for $\mathcal{I}(C_f^\emph{can})$ is given by the single quartic $\mathcal{F}_{4,w}$ with $\left(\Delta^{(1)}\right)^\circ\cap\left(\frac{1}{3}\mathbb{Z}\right)^2=\{w\}$.
  \item If $\Delta^{(1)}\cong 2\Sigma$ then a minimal set of generators for $\mathcal{I}(C_f^\emph{can})$ is 
  given by a basis for $\mathcal{I}_2(\emph{Tor}(\Delta^{(1)}))$ and the three cubics $\mathcal{F}_{3,w},\mathcal{F}_{3,w'},\mathcal{F}_{3,w''}$ with $\left(\Delta^{(1)}\right)^\circ\cap\left(\frac{1}{2}\mathbb{Z}\right)^2=\{w,w',w''\}$.
  \item In the other cases a minimal set of generators for the ideal $\mathcal{I}(C_f^\emph{can})$ is given by
  a basis for $\mathcal{I}_2(\emph{Tor}(\Delta^{(1)}))$ and the $g-3$ cubics $\mathcal{F}_{3,w}$ with $w\in \left(\Delta^{(1)}\right)^\circ\cap\left(\frac{1}{2}\mathbb{Z}\right)^2$.
\end{itemize}
%$\Delta = \Delta(f)$ and assume that $\Delta^{(1)}$ is two-dimensional. Then $U(f)$ is non-hyperelliptic
%and the Zariski closure $C$ of the image of (\ref{canonicalembedding}) is a canonical model of $U(f)$.
%For all integers $d\geq 2$, the equality (\ref{claimId}) holds. 
\end{theorem}

\begin{proof} 
%Recall from Section~\ref{section_introduction} that a lattice polygon $\Gamma$ is of the form $\Delta^{(1)}$ if and only if $\Gamma^{(-1)}$ is also a lattice polygon. This implies that $\partial \Delta^{(1)} \cap \mathbb{Z}^2 \geq 4$ (here we use the assumption $\Delta^{(1)}\not\cong\Upsilon$). Now we can use Koelman's criterion: it guarantees that the ideal of $\text{Tor}(\Delta^{(1)})$ is generated by quadrics.
From \cite[Thm.~8.1]{CaCo1}, the assumptions $\Delta^{(2)}\neq \emptyset$ and $\Delta^{(1)} \not\cong \Upsilon$ imply that the Clifford index of $C_f$ is at least $2$. 
In this case Petri's
theorem \cite{saintdonat} guarantees that $\mathcal{I}(C_f^\text{can})$ is generated by quadrics and the statement follows from (\ref{claimId}).\\

\noindent As for the other cases:
\begin{itemize}
  \item If $\Delta^{(1)} \cong \Upsilon$, the claim follows by noting that $\text{Tor}(\Upsilon)$ is cut out by the cubic $X_{-1,-1}X_{1,0}X_{0,1} - X_{0,0}^3$ and that a canonical curve of genus $4$ is of degree $6$, so that a single (necessarily unique) quadric suffices.
  \begin{center}
    \includegraphics[scale=0.9]{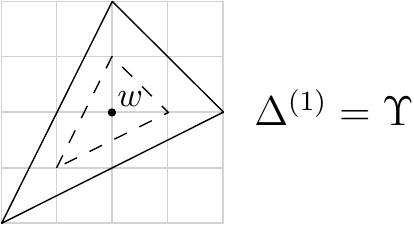}
  \end{center}
  \item The case $\Delta^{(1)}\cong\Sigma$ corresponds to smooth plane quartics and is obvious. 
  \item If $\Delta^{(1)} \cong 2\Sigma$ then $C_f$ is a smooth plane quintic. By Petri's theorem 
  we know that $\mathcal{I}(C_f^\text{can})$ is generated by quadrics and cubics. 
  Since $\mathcal{I}(\text{Tor}(\Delta^{(1)})$ is generated by quadrics, 
  the statement follows from (\ref{claimId}). (Note that $\text{Tor}(\Delta^{(1)})$ is just the Veronese surface.)
  \begin{center}
    \includegraphics[scale=0.9]{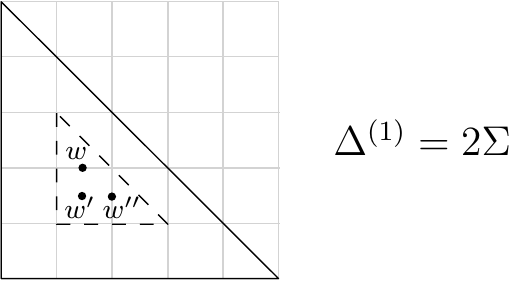}
  \end{center}
  \item In the other cases $C_f$ is a trigonal curve and $\sharp \left( \partial \Delta^{(1)} \cap \mathbb{Z}^2 \right) \geq 4$, so
  that $\text{Tor}(\Delta^{(1)})$ is generated by quadrics.
  By Petri's theorem 
  we know that $\mathcal{I}(C_f^\text{can})$ is generated by quadrics and cubics, 
  so that the statement again follows from (\ref{claimId}). (Note that
  $\text{Tor}(\Delta^{(1)})$ is a rational normal surface scroll.) Remark that
  \[ \sharp \left( \left(\Delta^{(1)}\right)^\circ\cap\left(\frac{1}{2}\mathbb{Z}\right)^2 \right) = \sharp \left( \left( 2 \Delta^{(1)} \right)^{(1)} \cap \mathbb{Z}^2 \right) = g - 3 \]
  by Pick's theorem.
  \begin{center}
    \includegraphics[scale=0.9]{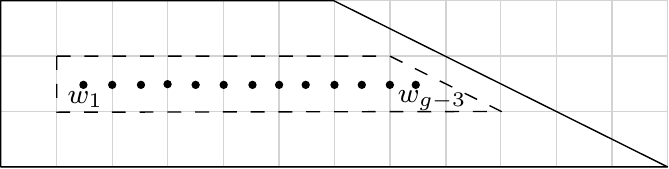}
  \end{center}
  
\end{itemize}
This concludes the proof.
\end{proof}

%\noindent \emph{Remarks.}
%\begin{itemize}
%\item The above proof shows that $\chi$ maps any element of $\mathcal{I}_2(C)$
%to $g f$, where $g \in k[x^{\pm 1},y^{\pm 1}]$ satisfies $\Delta(g) \subset \Delta^{(2)}$.
%%\item The proof of Theorem~\ref{cliffordtheorem}, to which we made reference, does not rely on this explicit description of the canonical ideal.
%\item The upper bound $g \leq \sharp (\Delta^{(1)} \cap \mathbb{Z}^2)$ holds
%in general (that is, regardless of the non-degeneracy condition) and was discovered by H.\ Baker in 1893 already \cite{HBaker}; see \cite{beelen} for
%an elementary proof of Baker's bound. Khovanskii was the first to give a sufficient condition for equality \cite[\S4.Ass.~2]{Khovanskii}.
%\end{itemize}

\noindent We remark that in the last case of trigonal curves, the generators $\mathcal{F}_{3,w}$ are just the `rolling factors'
that were introduced by Reid; see \cite{Stevens}. For more general polygons, 
our forms $\mathcal{F}_{d,w}$ can be viewed as analogues of these, where the `rolling' is done in two directions instead of one.\\

\noindent Theorem~\ref{thmcanonical} immediately gives rise to an efficient
algorithm for computing a minimal set of generators for the canonical ideal of $C_f$, for a given Laurent polynomial
$f \in k[x^{\pm 1}, y^{\pm 1}]$ that is non-degenerate with respect to
its Newton polygon $\Delta(f)$. As before we assume that 
\begin{itemize}
\item $\sharp (\Delta(f)^{(1)} \cap \mathbb{Z}^2) \geq 3$, so that $C_f$ is of genus $g \geq 3$, and
\item $\Delta(f)^{(1)}$ is two-dimensional, so that $C_f$ is non-hyperelliptic, or equivalently that
its Clifford index is at least $1$ (otherwise
the canonical image is just a rational normal curve).
\end{itemize}
In case $\Delta(f)^{(1)} \cong \Sigma$ the output consists of a single quartic. If not, 
it consists of independent quadratic and cubic generators
of the canonical ideal, i.e.\
${g-2 \choose 2}$ quadrics and $g-3$ cubics in the case of Clifford index $1$, and
just ${g-2 \choose 2}$ quadrics in the case of Clifford index at least $2$. 
Indeed, all one needs to do is adding the appropriate $\mathcal{F}_{d,w}$'s
to a minimal set of generators for $\text{Tor}(\Delta^{(1)})$. 
Finding these $\mathcal{F}_{d,w}$'s boils down to finding relations of the form (\ref{normality}), which
can be done by exhaustive search. 
An implementation can be found in the Magma file \texttt{canonicalideal.m}
that accompanies this paper. The function of interest is called \texttt{NondegIdeal()}.\\

\noindent \emph{Example.} The following sample code computes the canonical ideal of a genus $14$ curve in a fraction of a second:\\
\noindent \texttt{> load "canonical.m"}\\
\noindent \texttt{Loading "canonical.m"}\\
\noindent \texttt{Loading "basic\char`_commands.m"}\\
\noindent \verb"> R<x,y> := PolynomialRing(Rationals(),2);"\\
\noindent \verb"> f :=  13*x^6*y^5 - 6*x^6*y^4 + 2*x^3*y^5 + 4*x^3*y^4 + x^3 + 3*y^4;"\\
\noindent \verb"> AA := AffineSpace(Rationals(),2);"\\
\noindent \verb"> C := Curve(AA,f);"\\
\noindent \verb"> Genus(C);"\\
\noindent \verb"14"\\
\noindent \verb"> time I := Ideal(NondegIdeal(f));"\\
\noindent \verb"Time: 0.130"\\
\noindent In sharp contrast, it takes the Magma intrinsic  
way over an hour.\\
\noindent \verb"> time I := Ideal(Image(CanonicalMap(C)));"\\
\noindent \verb"Time: 5405.360"\\
\noindent Note moreover that in the latter case, in general, the output does not consist of a minimal
set of generators.\\

\noindent \emph{Remark.} Here again, the method can be slightly improved by taking into account the corresponding remark
from Section~\ref{section_toricideal}, i.e.\
by computing a set of generators for $\mathcal{I}(\text{Tor}(\Delta^{(1)}))$ using Gr\"obner
bases. It is also possible to do this at once for the entire ideal
$\mathcal{I}(C_f^\text{can})$, as below (continuation of the above example):\\
\noindent \verb"> latticepoints := ConvexHull(InnerPoints(NewtonPolytope(f)));"\\ 
\noindent \verb"> g := #latticepoints;"\\
\noindent \verb"> PP := ProjectiveSpace(Rationals(), g-1);"\\
\noindent \verb"> phi_can := map< C->PP | [x^p[1]*y^p[2] : p in latticepoints] >;"\\
\noindent \verb"> time I := Ideal(Image(phi_can));"\\
\noindent \verb"Time: 0.370"\\
This is already much faster than the Magma intrinsic, but slower than the previous method (the difference
in timing increases as the genus grows). Note again that the output does not necessarily
consist of a minimal set of generators.

\noindent \textsc{Departement Wiskunde, KU Leuven}\\
\noindent \textsc{Celestijnenlaan 200B, 3001 Leuven (Heverlee), Belgium}\\
\noindent \emph{E-mail address:} \verb"wouter.castryck@wis.kuleuven.be"\\

\noindent \textsc{Department of Mathematics and Applied Mathematics, University of Cape Town}\\
\noindent \textsc{Private Bag X1, Rondebosch 7701, South Africa}\\
\noindent \emph{E-mail address:} \verb"filip.cools@uct.ac.za"\\


\begin{thebibliography}{99}
\small

 %\bibitem{Arbarello} Enrico Arbarello, Maurizio Cornalba, Phillip Griffiths, Joe Harris, \emph{Geometry of algebraic curves, vol.\ I}, Grundlehren der mathematischen Wissenschaften \textbf{267}, Springer (1985)
% \bibitem{HBaker} Henry Baker,  \emph{Examples of applications of Newton's polygon to the theory
% of singular points of algebraic functions}, Transactions of the Cambridge
% Philosophical Society \textbf{15}, pp.\ 403-450 (1893)
 %\bibitem{baker} Matthew Baker, \emph{Specialization of linear systems from curves to graphs}, Algebra \& Number Theory \textbf{2}(6), pp.\ 613-653 (2008)
% \bibitem{beelen} Peter Beelen, \emph{A generalization of Baker's theorem}, Finite Fields and Their
%Applications \textbf{15}(5), pp.\ 558-568 (2009)
%\bibitem{Berchtold} Florian Berchtold, \emph{Lifting of morphisms to quotient presentations},  Manuscripta Mathematica \textbf{110}(1), pp.\ 33-44 (2003)
\bibitem{magma} Wieb Bosma, John Cannon, Catherine Playoust, \emph{The Magma algebra system. I. The user language}, Journal of 
Symbolic Computation \textbf{24}, pp.\ 235-265 (1997)
\bibitem{bruns} Winfried Bruns, Joseph Gubeladze, Ng$\hat{\text{o}}$ Vi$\hat{\text{e}}$t Trung, \emph{Normal polytopes, triangulations, and Koszul algebras}, 
Journal f\"ur die Reine und Angewandte Mathematik \textbf{485}, pp.\ 123-160 (1997)
%\bibitem{brunsgub} Winfried Bruns, Joseph Gubeladze, \emph{Polytopal linear groups}, Journal of Algebra \textbf{218}(2), pp.\ 715-737 (1999)
 %\bibitem{Brunsetal} Winfried Bruns, Joseph Gubeladze, Ng\^{o} Vi\^{e}t Trung, \emph{Normal polytopes, triangulations, and Koszul algebras}, Journal f\"ur die reine und angewandte Mathematik \textbf{405}, pp.\ 123-160 (1997) 
 %\bibitem{movingout} Wouter Castryck, \emph{Moving out the edges of a lattice polygon}, Discrete and Computational Geometry \textbf{47}(3), pp.\ 496-518 (2012)
% \bibitem{CaCo} Wouter Castryck, Filip Cools, \emph{Newton polygons and curve gonalities}, Journal of Algebraic Combinatorics \textbf{35}(3), pp.\ 345-366 + err.\ pp.\ 367-372 (2012)
 \bibitem{CDV} Wouter Castryck, Jan Denef, Fr\'ederik Vercauteren, \emph{Computing zeta functions of nondegenerate curves},
 International Mathematics Research Papers Vol.\ \textbf{2006}, Article ID 72017, pp.\ 1-57 (2006)
\bibitem{CaCo1} Wouter Castryck, Filip Cools, \emph{Linear pencils encoded in the Newton polygon}, preprint
 %\bibitem{caco_upcoming} Wouter Castryck, Filip Cools, \emph{On the intrinsicness of the Newton polygon}, in preparation
 %\bibitem{CastryckVoight} Wouter Castryck, John Voight, \emph{On nondegeneracy of curves}, Algebra \& Number Theory \textbf{3}(3), pp.\ 255-281 (2009)
% \bibitem{coppenskeemmartens} Marc Coppens, Changho Keem, Gerriet Martens, \emph{Primitive linear series on curves}, Manuscripta Mathematica \textbf{77}, pp.\ 237-264 (1992)
 %\bibitem{Coppens} Marc Coppens, \emph{The number of linear systems computing the gonality}, Journal of the Korean Mathematical Society \textbf{37}(3), pp.\ 437-454 (2000)
 %\bibitem{CoppensMartens} Marc Coppens, Gerriet Martens, \emph{Secant spaces and Clifford's theorem}, Compositio Mathematica \textbf{78}(2), pp.\ 193-212 (1991)
 \bibitem{coxlittleschenck} David Cox, John Little, Hal Schenck, \emph{Toric varieties}, Graduate Studies in Mathematics \textbf{124}, American Mathematical Society (2011)
 %\bibitem{DVCab} Jan Denef, Fr\'ederik Vercauteren, \emph{Computing zeta functions of $C_{ab}$ curves using Monsky-Washnitzer cohomology}, Finite Fields and Their Applications \textbf{12}(1), pp.\ 78-102 (2006)
 %\bibitem{poladjth} Sandra Di Rocco, Christian Haase, Benjamin Nill, Andreas Paffenholz, \emph{Polyhedral adjunction theory}, preprint (2011)
 %\bibitem{draismaetal} Jan Draisma, Tyrrell McAllister, Benjamin Nill, \emph{Lattice width directions and Minkowski's $3^d$-theorem}, SIAM Journal on Discrete Mathematics \textbf{26}(3), pp.\ 1104-1107 (2012)
\bibitem{eisenbud} David Eisenbud, \emph{The geometry of syzygies}, Graduate Texts in Mathematics \textbf{229}, Springer (2005)
% \bibitem{EiHa} David Eisenbud, Joe Harris, \emph{On Varieties of Minimal Degree (A Centennial Account)}, Proceedings in Symposia in Pure Mathematics \textbf{46}, pp.\ 3-13 (1987)
 %\bibitem{ELMS} David Eisenbud, Herbert Lange, Gerriet Martens, Frank-Olaf Schreyer, \emph{The Clifford dimension of a projective curve}, Compositio Mathematica \textbf{72}(2), pp.\ 173-204 (1989)
 %\bibitem{FejesMakai} L\'aszl\'o Fejes T\'oth, Endre Makai Jr., \emph{On the thinnest non-separable lattice of convex plates}, Studia Scientiarum Mathematicarum Hungarica \textbf{9}, pp.\ 191-193 (1974)
 %\bibitem{fulton} William Fulton, \emph{Introduction to toric varieties}, Annals of Mathematics Studies \textbf{131}, Princeton University Press (1993)
 %\bibitem{GKZ} Israel Gelfand, Mikhail Kapranov, Andrei Zelevinsky, \emph{Discriminants, resultants, and multidimensional determinants}, Birkh\"auser Boston (1994)
\bibitem{HaaseSchicho} Christian Haase, Josef Schicho, \emph{Lattice polygons and the number $2i+7$}, American Mathematical Monthly \textbf{116}(2), pp.\ 151-165 (2009)
 %\bibitem{cantor} Ryuichi Harasawa, Joe Suzuki, \emph{Fast Jacobian group arithmetic on $C_{ab}$ curves}, Proceedings of ANTS-IV (Leiden, The Netherlands), Lecture Notes in Computer Science \textbf{1838}, pp.\ 359-376 (2000)
 %\bibitem{harris} Joe Harris, \emph{Algebraic geometry: a first course}, Graduate Texts in Mathematics \textbf{133}, Springer (1992)
% \bibitem{hartshorne} Robin Hartshorne, \emph{Algebraic geometry}, Graduate Texts in Mathematics \textbf{52}, Springer (1977)
 %\bibitem{Harui} Takeshi Harui, \emph{The gonality and the Clifford index of curves on an elliptic ruled surface}, Archiv der Mathematik \textbf{84}, pp.\ 131-147 (2005)
 %\bibitem{kawgon1} Ryo Kawaguchi, \emph{The gonality conjecture for curves on certain toric surfaces}, Osaka Journal of Mathematics \textbf{45}, pp.\ 113-126 (2008)
% \bibitem{Kawaguchi2} Ryo Kawaguchi, \emph{Weierstrass gap sequences on curves on toric surfaces}, Kodai Mathematical
%    Journal \textbf{33}, pp.\ 63-86 (2010)
 %\bibitem{kawgon2} Ryo Kawaguchi, \emph{The gonality conjecture for curves on toric surfaces with two $\mathbb{P}^1$-fibrations}, Saitama Mathematical Journal \textbf{27}, pp.\ 35-80 (2010)
 %\bibitem{Kawaguchi} Ryo Kawaguchi, \emph{The gonality and the Clifford index of curves on a toric surface}, preprint
 \bibitem{Hering} Milena Hering, \emph{Syzygies of toric varieties}, Ph.D.\ thesis, University of Michigan (2006)
 \bibitem{Hess} Florian Hess, \emph{Computing Riemann-Roch spaces in algebraic function fields and related topics}, Journal of Symbolic Computation \textbf{33}(4), pp.\ 425-445 (2002) 
 \bibitem{Khovanskii} Askold G.\ Khovanskii, \emph{Newton polyhedra and toroidal varieties},
Functional Analysis and Its Applications \textbf{11}(4), pp.\ 289-296 (1977)
 \bibitem{Koelman} Robert J.\ Koelman, \emph{The number of moduli of families of curves on toric surfaces},
 Ph.D. thesis, Katholieke Universiteit Nijmegen (1991)
 \bibitem{Koelman1} Robert J.\ Koelman, \emph{Generators of the ideal of a projectively embedded toric surface}, 
 Tohoku Math.\ J.\ \textbf{45}(3), pp.\ 385-392 (1993)
 \bibitem{Koelman2} Robert J.\ Koelman, \emph{A criterion for the ideal of a projectively embedded toric surface to be generated by quadrics}, Beitr\"age zur Algebra und Geometrie \textbf{34}, pp.\ 57-62 (1993)
% \bibitem{Looijenga} Claudio Fontanari, Eduard Looijenga, \emph{A perfect stratification of $\mathcal{M}_g$ for $g \leq 5$}, Geometriae Dedicata \textbf{136}(1), pp.\ 133-143 (2008)
 %\bibitem{Loose} Frank Loose, \emph{On the graded Betti numbers of plane algebraic curves}, Manuscripta Mathematica \textbf{64}, pp.\ 503-514 (1989)
 %\bibitem{LubbesSchicho} Niels Lubbes, Josef Schicho, \emph{Lattice polygons and families of curves on rational surfaces}, Journal of Algebraic Combinatorics \textbf{34}, pp.\ 213-236 (2012)
 %\bibitem{Matsumoto} Ryutaroh Matsumoto, \emph{The $C_{ab}$ curve}, note available at \verb"http://www.rmatsumoto.org/cab.pdf"
% \bibitem{MRoig} Rosa M. Mir\'o-Roig, \emph{The represenation type of rational normal scrolls},
%Rendiconti del Circolo Matematico di Palermo \textbf{62}, pp. 153-164 (2013)
 %\bibitem{Miura} Shinji Miura, \emph{Error-correcting codes based on algebraic geometry}, Ph.D.\ thesis, University of Tokyo (1997)
 %\bibitem{Namba} Makoto Namba, \emph{Families of meromorphic functions on compact Riemann surfaces}, Lecture Notes in Mathematics \textbf{767}, Springer-Verlag (1979)
%\bibitem{Sagraloff} Michael Sagraloff, \emph{Special linear series and syzygies of canonical curves of genus $9$}, Ph.D.\ thesis, Universit\"at des Saarlandes (2008)
 \bibitem{saintdonat} Bernard Saint-Donat, \emph{On Petri's analysis of the linear system of quadrics through a canonical curve}, Mathematische Annalen \textbf{206}, pp.\ 157-175 (1973)
%\bibitem{Serrano} Fernando Serrano, \emph{Extensions of morphisms defined on a divisor}, Mathematische Annalen \textbf{277}, pp. 395-413 (1987)
%\bibitem{Schenck} Hal Schenck, \emph{Lattice polygons and Green's theorem}, Proceedings of the American Mathematical Society \textbf{132}(12), pp.\ 3509-3512 (2004)
% \bibitem{Schreyer} Frank-Olaf Schreyer, \emph{Syzygies of Canonical Curves and Special Linear Series}, Mathematische Annalen \textbf{275}, pp.\ 105-137 (1986)
 \bibitem{Stevens} Jan Stevens, \emph{Rolling factors deformations and extensions of canonical curves}, Documenta Mathematica \textbf{6}, pp.\ 185-226 (2001) 
\end{thebibliography}
\end{document}